\documentclass[11pt]{article}

\usepackage{authblk}

\usepackage{fullwidth}
\usepackage{color}
\usepackage[color,matrix,arrow]{xy}
\usepackage[top=1.5in, bottom=1.5in, left=1in, right=1in]{geometry}
\usepackage{amsgen}
\usepackage{amsmath}
\usepackage{amstext}
\usepackage{sectsty}
\usepackage{amsbsy}
\usepackage{amsopn}
\usepackage{amsfonts}
\usepackage{amssymb}
\usepackage{eepic}
\usepackage{graphicx}
\usepackage{epsf}
\usepackage{pstricks}
\xyoption{all}
\usepackage{amsthm}

\def\A{{\cal{A}}}

\def\iff{\Leftrightarrow}
\def\Rw{\Rightarrow}
\def\oo{\overline}
\def\wt{\widetilde}
\def\wh{\widehat}

\def\N{\mathbb{N}}

\def\fix{\mbox{Fix}}

\def\pref{\mbox{Pref}\,}

\def\max{\mbox{max}}
\def\min{\mbox{min}}

\def\rk{\mbox{rk}}

\def\id{\mbox{id}}
\def\sup{\mbox{sup}}

\def\Rat{\mbox{Rat}}

\def\st{{\rm St}}

\def\p{\varphi}

\def\inv{^{-1}}

\def\bi{\begin{itemize}}
\def\ei{\end{itemize}}
\def\beq{\begin{equation}}
\def\eeq{\end{equation}}

\def\xr{\xrightarrow}

\newtheorem{T}{Theorem}[section]
\newcommand{\bt}{\begin{T}}
\newcommand{\et}{\end{T}}
\newcommand{\ftd}{$\square$\end{T}}
\newcommand{\Eq}{\text{Eq}}
\newcommand{\Ker}{\text{Ker}}
\newcommand{\Fix}{\text{Fix}}
\newcommand{\End}{\text{End}}
\newcommand{\Aut}{\text{Aut}}
\newcommand{\Mon}{\text{Mon}}

\newtheorem{Proposition}[T]{Proposition}
\newcommand{\bp}{\begin{Proposition}}
\newcommand{\ep}{\end{Proposition}}
\newcommand{\fpd}{$\square$\end{Proposition}}

\newtheorem{Lemma}[T]{Lemma}
\newcommand{\bl}{\begin{Lemma}}
\newcommand{\el}{\end{Lemma}}
\newcommand{\fld}{$\square$\end{Lemma}}

\newtheorem{Corol}[T]{Corollary}
\newcommand{\bc}{\begin{Corol}}
\newcommand{\ec}{\end{Corol}}
\newcommand{\fcd}{$\square$\end{Corol}}

\newtheorem{Result}[T]{Result}
\newcommand{\br}{\begin{Result}}
\newcommand{\er}{\end{Result}}
\newcommand{\frd}{$\square$\end{Result}}

\newtheorem{Example}[T]{Example}
\newcommand{\be}{\begin{Example}}
\newcommand{\ee}{\end{Example}}

\newtheorem{Problem}[T]{Problem}
\newcommand{\bq}{\begin{Problem}}
\newcommand{\eq}{\end{Problem}}

\newtheorem{Remark}[T]{Remark}
\newcommand{\brm}{\begin{Remark}}
\newcommand{\erm}{\end{Remark}}

\newtheorem{conjecture}[T]{Conjecture}

\begin{document}

\title{On fixed points and equalizers of injective endomorphisms of the free group at infinity}

\author[1]{Andr\'e Carvalho}
\affil[1]{Research Center in Mathematics and Applications (CIMA)
\vspace{-0.2cm}

Department of Mathematics, School of Sciences and Technology of the University of \'Evora

	Rua Rom\~ao Ramalho, 59, 7000-671 \'Evora, Portugal
	
	\texttt{andrecruzcarvalho@gmail.com}\\}
	
\author[2]{Pedro V. Silva}
\affil[2]{Centre of Mathematics of the University of Porto, 
	
	Department of Mathematics,
	Faculty of Sciences of the University of Porto
	
	R. Campo Alegre 687, 4169-007 Porto, 
	Portugal
	
	\texttt{pvsilva@fc.up.pt}}

\maketitle

\begin{abstract}
We study equalizers and fixed points of monomorphisms of free groups at infinity. We show that the action of the equalizer of two monomorphisms on the regular points of the equalizer at infinity has finitely many orbits, showing that the equalizer at infinity is, in some sense, finitely generated and generalizing a previous result of Cooper about fixed points. We additionally show that it is decidable whether an automorphism of a free group has a nontrivial fixed point at infinity. The same result is shown for monomorphisms satisfying the condition of being almost length-increasing. We remark that almost length-increasing  monomorphisms are generic among endomorphisms of a free group. We also prove that being almost length-increasing is a decidable condition. We end the paper with several open problems arising from this work.
\end{abstract}

\section{Introduction}
The study of fixed subgroups of endomorphisms of groups started with the (independent) work of Gersten \cite{[Ger87]} and Cooper \cite{[Coo87]},  proving that the subgroup of fixed points $\Fix(\phi)=\{x\in F_n\mid x\phi=x\}$ of any automorphism $\phi$ of a free group $F_n$ is  finitely generated.
Scott conjectured that the rank of such a fixed subgroup could only be at most $n$ and this was successfully proved by Bestvina and Handel in 1992 \cite{[BH92]}. Imrich and Turner extended these results for general endomorphisms of the free group, by reducing it to the automorphism case in \cite{[IT89]}. The problem of computing a basis for $\Fix(\phi)$ was settled by Bogopolski and Maslakova in 2016 in \cite{[BM16]} for automorphisms and by Mutanguha \cite{[Mut22]} for general endomorphisms of a free group. Hence, fixed subgroups of endomorphisms of free groups are now (in some sense) well understood.  We refer the reader to \cite{[Ven02]} for a survey on fixed subgroups in free groups.
Results about fixed subgroups of different classes of groups have since then been obtained by many (see, for example \cite{[Pau89],[RSS13], [MS21]}). 

However, equalizers are not so well understood. In \cite{[Sta83]}, Stallings conjectured that, given two finitely generated free group homomorphisms, $\phi,\psi:F_n\to F_m$, with injective $\phi$, the equalizer $\Eq(\phi,\psi)=\{x\in F_n\mid x\phi=x\psi\}$ would be finitely generated and its rank bounded by $n$ (see also \cite[Problem 6]{[DV96]}, \cite[Conjecture 8.3]{[Ven02]}, \cite[Problem F31]{[BMS02]}). It was proved in \cite{[GT89]} that the equalizer was always finitely generated, but, so far, the conjecture remains open as there is no bound on its rank. For the free group of rank 2, the problem was solved by Logan in \cite{[Log22]}.
It is clear that equalizers generalize fixed subgroups as $\Fix(\phi)=\Eq(\phi, \id)$ and that in case one of the homomorphisms, say $\psi$, is  bijective, the notions coincide because $\Eq(\phi, \psi)=\Fix(\phi\psi\inv)$. 

Post Correspondence Problem (PCP) is a classical problem in computer science and it can be stated as the problem of deciding, given as input two homomorphisms between free monoids, $\phi,\psi:X^*\to Y^*$, whether the equalizer $\Eq(\phi,\psi)=\{x\in X^*\mid x\phi=x\psi\}$ is trivial or not. Its undecidability was proved by Post in 1946 \cite{[Pos46]}.

Naturally, one can think of the analogous question for  groups. Let $A,B$ be alphabets. The Post Correspondence Problem for free groups is the problem of deciding, given homomorphisms $\phi,\psi:F_A\to F_B$, whether $\Eq(\phi,\psi)$ is trivial or not. While there are results for other classes of groups \cite{[MNU14],[CLL24]}, the problem for free groups is an important open question (see \cite[Problem 5.1.4]{[DKLM19]}
 and \cite[Section 1.4]{[MNU14]}). We refer the reader to \cite{[CL21]} for a survey on the PCP and similar questions for free groups.
 
We will focus on the case of nonabelian free groups. Composing with an injective homomorphism $\theta:F_m\to F_n$ on the right, we have that $\Eq(\phi,\psi)$ is nontrivial if and only if $\Eq(\phi\theta,\psi\theta)$ is nontrivial.  So, we will assume that both ranks are the same and talk about endomorphisms instead of homomorphisms.
We know that triviality (and even computability) of the fixed subgroup is decidable \cite{[Mut22]} and it is easy to see that if $\phi$ and $\psi$ are both noninjective, then $\Ker(\phi)\cap \Ker(\psi)$ is always nontrivial 
(if $u \in \Ker(\phi)$ and $v \in \Ker(\psi)$ are nontrivial, take a common power if they commute and their commutator if they do not),
so the PCP is trivially decidable in this case. However, the case where one homomorphism is injective and the other one is not bijective remains unknown.

For monoids, the infinite variation of the PCP asks whether there is a right-infinite word where the two morphisms coincide and it is shown to be undecidable \cite{[DL12]} and ``not more complex than the PCP"  \cite{[Fin15]}. We will study equalizers and fixed points of endomorphisms of the free group at infinity.

Given two words $u$ and $v$ on a free group, we write $u\wedge v$ to denote the longest common prefix of $u$ and $v$. The prefix metric on a free group is defined by
$$d(u,v)=\begin{cases}
2^{-|u\wedge v|} \text{ if $u\neq v$}\\
0 \text{ otherwise}
\end{cases}.$$
The  completion $(\widehat{F}_n,\widehat{d})$ is a compact space which can be described as the set of all finite and infinite reduced words on the alphabet $A \cup A^{-1}$. We will denote by $\partial F_n$ the set consisting of only the infinite words and call it the \emph{boundary} of $F_n$. This coincides with the Gromov boundary defined more generally for hyperbolic groups.

It is well known by a general topology result \cite[Section XIV.6]{[Dug78]} that every uniformly continuous mapping $\varphi$ between metric spaces admits a unique continuous extension $\widehat\varphi$ to the completion. The converse is obviously true in this case by compactness. For free groups, uniformly continuous endomorphisms are precisely the injective ones. So, for this purpose, we will only consider equalizers and fixed points for injective endomorphisms.

When considering the continuous extension $\wh \phi$ of an injective endomorphism $\phi\in \End(F_n)$ to the completion $\wh F_n$, all elements in the topological closure $\Fix(\phi)^c$ of $\Fix(\phi)$ are fixed by $\wh \phi$. But these might not be the only ones. We will say that elements belonging to $\Fix(\phi)^c$ are singular (infinite) fixed points and elements in $\Fix(\wh\phi)\setminus \Fix(\phi)^c$ are regular (infinite) fixed points.
In \cite{[Coo87]}, Cooper proved that the fixed subgroup of a free group automorphism $\phi$ is finitely generated and that, when considering the extension $\widehat\phi$ of $\phi$ to the Gromov boundary, then the natural action of $\Fix(\phi)$ on the regular infinite fixed points of $\wh \phi$ has finitely many orbits, showing that $\Fix(\widehat\phi)$ is, in some sense, finitely generated.

We show that Cooper's techniques for fixed subgroups of automorphisms can be adapted to the setting of equalizers of injective endomorphisms of the free group, by proving that the action of $\Eq(\phi,\psi)$ on the regular points of $\Eq(\wh \phi,\wh\psi)$ has finitely many orbits, showing that, in some sense, the equalizer \emph{at infinity} is also finitely generated. We remark that, in general, equalizers are much harder to understand than fixed subgroups.

\newtheorem*{eqfingen}{Theorem \ref{eq_fin_gen}}
\begin{eqfingen}
Let $\phi,\psi$ be injective endomorphisms of $F_n$. Then 
the action of $\Eq(\phi,\psi)$ on the regular points of $\Eq(\wh\phi,\wh\psi)$ has finitely many orbits.
\end{eqfingen}

As a corollary, we have that the PCP answers yes on input $(\phi,\psi)$ if and only if $\Eq(\wh\phi,\wh\psi)$ is finite and we pose the infinite version of Stallings conjecture, which asks for a bound on the number of orbits of the action.

We also study decidability of the triviality of fixed points \emph{at infinity}, meaning the decidability of the following algorithmic problem: given an injective endomorphism $\phi\in \End(F_n)$ as input, can we decide if $\Fix(\wh\phi)$ is trivial? One of the motivations is the PCP at infinity. In the finite case, for injective $\phi,\psi\in \Mon(F_n)$, it is not known if triviality of $\Eq(\phi,\psi)$ is decidable, unless one of the monomorphisms is an automorphism, in which case the PCP simply asks for the decidability of triviality of a certain fixed subgroup. We start by proving the result for automorphisms.

\newtheorem*{trivialfix}{Theorem \ref{trivialfix}}
\begin{trivialfix}
It is decidable, for a given $\phi \in {\rm Aut}(F_n)$, whether or nor ${\rm Fix}(\wh{\phi})$ is trivial.
\end{trivialfix}

The (asymptotic) density of a subset $S\subseteq F_n$ of a free group with basis $A=\{a_1,\ldots, a_n\}$ is defined as $$D(S)=\lim_{n\to \infty}\frac{|S\cap B_A(n)|}{|B_A(n)|},$$ where $B_A(n)$ denotes the ball of radius $n$ centered at $1$ in the Cayley graph of $F_n$ with respect to $A$. Similarly, the density of a set $S\subseteq F_n^k$ of $k$-tuples of elements of $F_n$ can be defined as 
$$D(S)=\lim_{n\to \infty}\frac{|S\cap \left(B_A(n)\right)^k|}{|\left(B_A(n)\right)^k|}$$

Combinatorially, we can view an endomorphism of $F_n$ as the $n$-tuple $(a_1\phi,\ldots, a_n\phi)$ of elements of $F_n$. Hence, we can define the density of a subset $S\subseteq \End(F_n)$ of endomorphisms as the density of the set of $n$-tuples defined by $S$.
A set is said to be generic if its asymptotic density is $1$. 

We will then turn our attention to almost length-increasing monomorphisms, that is, monomorphisms  $\p$ for which  $|u\p| < |u|$ for only finitely many $u \in F_n$.  Wagner proved in \cite{[Wag99]} that endomorphisms with remnant are  generic in free groups (see also \cite{[Sta10]} for a strengthening). An endomorphism $\phi$ of $F_n$ is said to have remnant with respect to a fixed basis $\{a_1,\ldots, a_n\}$ of $F_n$ if, for each $i\in\{1,\ldots, n\}$, there is a nontrivial factor of $a_i\phi$ that does not cancel in any product of the form 

$$(a_j\phi)^{\alpha_j}(a_i\phi)(a_k\phi)^{\beta_k},$$
where $j,k\in \{1,\ldots, n\}$, $\alpha_l,\beta_l\in\{-1,0,1\}$, for $l\neq i$ and $\alpha_i,\beta_i\in\{0,1\}$.

It is easy to see that endomorphisms with remnant are, in particular, length-increasing, and that almost length-increasing endomorphisms must be injective (if not, there are infinitely many elements in the kernel whose length reduces when applying the endomorphism), thus showing that length-increasing (and so, almost length-increasing) monomorphisms are generic among endomorphisms of the free group.

Under the assumption that our endomorphism is almost length-increasing, we not only show that triviality of $\Fix(\wh\p)$ is decidable, but $\Fix(\wh\p)$ can, in some sense,  be computed.

\newtheorem*{alifix}{Theorem \ref{alifix}}
\begin{alifix}
Let $\p$ be an almost length-increasing monomorphism of $F_n$ with ${\rm Fix}(\p) = \{1\}$. Then ${\rm Fix}(\wh{\p})$ is computable.
\end{alifix}

We also show that, given an endomorphism of $F_n$, it is decidable whether it is an almost length-increasing monomorphism or not. We end the paper with some future directions and possible attacks on the case of general monomorphisms.

\section{Equalizers at infinity}
Let $\phi,\psi\in \Mon(F_n)$ be two injective endomorphisms of a nonabelian free group and $\wh\phi,\wh\psi$ be the corresponding continuous extensions to the completion $\wh{F_n}$. It is easy to see that $\Eq(\phi,\psi)$ acts on $\Eq(\wh\phi,\wh\psi)$ by multiplication on the left. 
 Elements that belong to the topological closure of $\Eq(\phi,\psi)$ are called \emph{singular}, and those belonging to $\Eq(\wh\phi,\wh\psi)\setminus (\Eq(\phi,\psi))^c$ are called \emph{regular}.
Also, it is easy to see that the orbit of a singular point contains only singular points (and so, the orbit of a regular point contains only regular points).
Notice that the action of $\Eq(\phi,\psi)$ on the singular elements of $\Eq(\wh\phi,\wh\psi)$ can have infinitely many orbits: for example, if $\phi=\psi=id$, then $\Eq(\wh\phi,\wh\psi)$ is uncountable.

The main goal of this section is to show that   the action of $\Eq(\phi,\psi)$ on the regular points of $\Eq(\wh \phi,\wh\psi)$ has finitely many orbits, showing that, in some sense, the equalizer \emph{at infinity} is finitely generated. Notice that $\Eq(\phi,\psi)$ is known to be finitely generated.

As it happened in \cite{[Coo87]}, where the analogous result for fixed points was proved, our main tool will be the \emph{bounded reduction property}.
An endomorphism of the free group has the bounded reduction property (BRP) if there is a constant $M>0$ such that, whenever $u,v\in F_n$ are such that the product $uv$ is reduced, then $|(uv)\phi| \geq |u\phi| + |v\phi| -2M$. In this case, $M$ is said to be a bounded reduction constant (or a BRP constant) for $\phi$. It is well known that an endomorphism of the free group has the BRP if and only if it is injective (see, for example, \cite{[Sil13]}). It is easy to see that the extension of an injective endomorphism of the free group  to the infinity also has the BRP in the following sense: there is a constant $M$ such that, for all $u\in F_n$ and $\alpha\in \partial F_n$, if the product $u\alpha$ is reduced, then there are at most $M$ cancellations occurring in $(u\phi)(\alpha\phi)$.

 Given $\alpha \in \partial F_n$ and $k \geq 0$, we denote by $\alpha^{[k]}$ the prefix of length $k$ of $\alpha$.

We start with two technical lemmas.

\bl \label{brp infinita}
Let $\varphi$ be an injective endomorphism of $F_n$. Then
$$\forall N\in \N \; \exists M\in \N\; \forall \alpha_1,\alpha_2\in \wh F_n\;\; |\alpha_1\wedge \alpha_2|\leq N \implies |\alpha_1\wh\varphi\wedge \alpha_2\wh\varphi| \leq M.$$
\el
\proof Let $B_\varphi$ be a bounded reduction constant for $\varphi$.
It follows from the bounded reduction property (see also \cite[Chapter 5, Proposition 15]{[GH90]}) that 
$$\forall N\in \N \; \exists M_N\in \N\; \forall u,v\in F_n\;\; |u\wedge v|\leq N \implies |u\varphi\wedge v\varphi| \leq M_N.$$
Let $N\in \N$ and $\alpha_1,\alpha_2\in \wh F_n$ be such that $|\alpha_1\wedge \alpha_2|\leq N$. If $\alpha_1,\alpha_2\in F_n$, we are done. So assume that at least one of the two, say $\alpha_1$, belongs to $\partial F_n$.
{Note that $(\partial F_n)\widehat{\varphi} \subseteq 
\partial F_n$ since $\varphi$ is injective.}
Take
$k\in \N$ such that $|\alpha_1^{[k]}\varphi|>|\alpha_1\wh\varphi\wedge\alpha_2\wh\varphi|+B_\varphi$.
Then, by the bounded reduction property, $\alpha_1\wh\varphi\wedge\alpha_2\wh\varphi$ is a prefix of $\alpha_1\wh\varphi\wedge\alpha_1^{[k]}\varphi$.

If $\alpha_2\in F_n$, then 
 $$|\alpha_1\wh\varphi\wedge \alpha_2\varphi|=|\alpha_1^{[k]}\varphi\wedge \alpha_2\varphi|\leq M_N,$$
since $|\alpha_1^{[k]}\wedge \alpha_2|\leq |\alpha_1\wedge\alpha_2|\leq N.$

If, on the other hand, $\alpha_2\in \partial F_n$, then we can choose $k'\in \N$ such that $|\alpha_2^{[k']}\varphi|>|\alpha_1\wh\varphi\wedge\alpha_2\wh\varphi|+B_\varphi$. and put $s=\max\{k,k'\}$. In this case,  $\alpha_1\wh\varphi\wedge\alpha_2\wh\varphi$ is a prefix of $\alpha_1\wh\varphi\wedge\alpha_1^{[s]}\varphi$ and of  $\alpha_2\wh\varphi\wedge\alpha_2^{[s]}\varphi$.
Thus,
 $$|\alpha_1\wh\varphi\wedge \alpha_2\varphi|=|\alpha_1^{[s]}\varphi\wedge \alpha_2^{[s]}\varphi|\leq M_N,$$
since $|\alpha_1^{[s]}\wedge \alpha_2^{[s]}|\leq |\alpha_1\wedge\alpha_2|\leq N.$
\qed\\

For each $N\in \N$, we will denote the constant $M$ coming from the previous lemma by {$M_N^\varphi$.}

\bl \label{prefixo igualizador}
Let $\phi$ and $\psi$ be injective endomorphisms of $F_n$. Then there is some constant $M$ such that, for all {distinct} $\alpha_1,\alpha_2\in \Eq(\wh\phi,\wh\psi)$, letting $\alpha=\alpha_1\wedge \alpha_2$, we have that $|(\alpha\inv\phi)(\alpha\psi)|\leq M$.
\el
\proof  
{Let $B$ be a bounded reduction constant for both $\phi$ and $\psi$.}
Let $\beta=\alpha_1\wh\phi\wedge\alpha_2\wh\phi  \; =\alpha_1\wh\psi\wedge\alpha_2\wh\psi$ and put $k=|\alpha\phi|-|\beta|$. We will start by showing that $|k|\leq \max\{M_0^\phi, B\}$.
For $i=1,2$, write $\alpha_i=\alpha\delta_i$ as a reduced word.
First, notice that
\begin{align*}
|\beta|=|\alpha_1\wh\phi\wedge\alpha_2\wh\phi|=|(\alpha\phi)(\delta_1\wh\phi)\wedge (\alpha\phi)(\delta_2\wh\phi)|\leq |\alpha\phi| + |\delta_1\wh\phi\wedge\delta_2\wh\phi|=
k + |\beta| + |\delta_1\wh\phi\wedge\delta_2\wh\phi|,
\end{align*} 
so $ |\delta_1\wh\phi\wedge\delta_2\wh\phi|\geq -k.$ 
 Since $\phi$ is injective, then the sequence $(M_n^\phi)_n$ is unbounded, so let
$N$ be the smallest natural number such that  $M_N^\phi>-k$. Then, we have that $|\delta_1\wedge\delta_2|\geq N$, since, otherwise, we would have $ |\delta_1\wh\phi\wedge\delta_2\wh\phi|\leq 
{M_{N-1}^\phi} 
\leq -k.$ But $\delta_1\wedge \delta_2=0$, by the maximality of $\alpha$. So, we deduce that $N\leq 0$, that is, that $N=0$, and $-k\leq M_0^\phi$.

Also, $$|\beta|\geq \min\{|\alpha\phi\wedge \alpha_1\wh\phi|,|\alpha\phi\wedge \alpha_2\wh\phi|\}\geq |\alpha\phi|-B =k+|\beta|-B,$$
so $k\leq B$. Hence $-M_0^{\phi} \leq |\alpha\phi|-|\beta| \leq B$ and also $-M_0^{\psi} \leq |\alpha\psi|-|\beta| \leq B$. Putting
$K=\max\{M_0^\phi, M_0^\psi, B\}$, we get that 
$$\big{|}
|\alpha\phi|-|\beta|
\big{|},
\big{|}
|\alpha\psi|-|\beta|
\big{|}
\leq K.$$
Since $\alpha_1 = \alpha\delta_1$ is a reduced product, it follows that $|\alpha\phi|-|\alpha\phi\wedge\alpha_1\wh\phi|\leq B$. Thus
$\big{|}
|\alpha\phi\wedge\alpha_1\wh\phi|-|\beta|
\big{|}
\leq K+B$. Since $\alpha\phi\wedge\alpha_1\wh\phi$ and $\beta$ are both prefixes of $\alpha_1\wh\phi$, then $\alpha\phi = \beta u$ for some $u \in F_n$ of length $\leq K+2B$. Similarly, $\alpha\psi = \beta v$ for some $v \in F_n$ of length $\leq K+2B$. Therefore $|(\alpha\inv\phi)(\alpha\psi)| = |u\inv v| \leq 2K+4B$ and so taking $M=2K+4B$ suffices.
\qed\\

We can now prove the main result of this section.

\bt \label{eq_fin_gen}
 Let $\phi,\psi$ be injective endomorphisms of $F_n$. Then 
the action of $\Eq(\phi,\psi)$ on the regular points of $\Eq(\wh\phi,\wh\psi)$ has finitely many orbits.
\et

\proof
 Write $E = \Eq(\phi,\psi)$. 
Suppose that the action of $E$ on the regular points of $\Eq(\wh\phi,\wh\psi)$ has infinitely many orbits, and let $\alpha_1,\alpha_2,\ldots$ be an infinite list of regular points of
 $\Eq(\wh\phi,\wh\psi)$ belonging to distinct orbits. Define a new infinite set of words by taking $\beta_i$ to be an element of $E\alpha_i$ maximizing $d(\beta_i,E)$, where $d$ is the prefix metric. 
 The $\beta_i$ are well defined since Im$(d) = \{ 2^{-m} \mid m \in \mathbb{N} \} \cup \{ 0\}$.
 
Since $\wh F_n$ is compact, 
 $(\beta_k)_k$ admits a convergent subsequence. By refining $(\beta_k)_k$, we can then assume that 
$\beta_{k}\to \gamma,$
  for some $\gamma\in \wh F_n$. Since $\wh\phi$ and $\wh\psi$ are both uniformly continuous, then $\wh\phi(\beta_{k})\to \wh\phi(\gamma)$ and $\wh\psi(\beta_{k})\to \wh\psi(\gamma)$. But, since $\beta_{k}\in \Eq(\wh\phi,\wh\psi)$, for all $k\in \N$, we have that $\wh\psi(\beta_{k})=\wh\phi(\beta_{k})\to \wh\phi(\gamma)$, and so $\gamma\in\Eq(\wh\phi,\wh\psi)$.
	 Note that $\gamma \in \partial F_n$ since every finite word is an isolated point.

Also, we may assume, refining further the sequence if necessary, that 
$|\beta_{k+1}\wedge\gamma| > 3|\beta_{k}\wedge\gamma|$ for every $k$.  Hence $\gamma$ is distinct from every $\beta_k$. Put $\delta_k=\beta_{k}\wedge \gamma$. Notice that, for  $i < j$, we have that $\delta_i$ is a proper prefix of $\delta_j$ and $|\delta_j| > 3|\delta_i|$.
Now, by Lemma \ref{prefixo igualizador}, there is a universal constant $M$ such that, for all $k\in \N$, there is some $\varepsilon_k\in F_n$ such that $|\varepsilon_k|\leq M$ and  $\delta_k\phi=(\delta_k\psi)\varepsilon_k$  (we apply the lemma to $\beta_k$ and $\gamma$). Hence, we can fix some $i,j\in \N$ such that $i < j$, $\varepsilon_i=\varepsilon_j$ 
 and $\delta_i,\delta_j$ end by the same letter in reduced form.

Put $\eta=\delta_i\delta_j\inv$. Notice that $$\eta\phi=(\delta_i\delta_j\inv)\phi
=\left((\delta_i\phi)\varepsilon_i\inv\right)\left(\varepsilon_i(\delta_j\inv\phi)\right)
=(\delta_i\psi)(\delta_j\inv\psi)=\eta\psi,$$
that is, $\eta\in E$.

It suffices to prove that $d(\beta_j,E) < d(\eta\beta_j,E)$, contradicting the choice of $\beta_j$ as an element of $E\alpha_i$ maximizing $d(\beta_j,E)$. It follows from the definition of $d$ that this is equivalent to having
\beq
\label{froholdt1}
\sup\{|\eta\beta_{j}\wedge u|\; : u \in E) \} < \sup\{|\beta_{j}\wedge u|\; : u \in E) \}.
\eeq

Write $\beta_j = \delta_j\beta'$. Then $\eta\beta_j = \delta_i\beta'$. Since $\beta_j = \delta_j\beta'$ is a reduced product and $\delta_i,\delta_j$ end by the same letter, then $\delta_i\beta'$ is also reduced. Since $\eta\beta_j$ is regular, it is not arbitrarily close to points in $E$. Hence we may define
$$m = \max\{|\delta_i\beta' \wedge u|\; : u \in E \} = \sup\{|\eta\beta_j \wedge u|\; : u \in E \}.$$
Fix $v \in E$ such that $|\delta_i\beta'\wedge v| = m$. It suffices to show that
\beq
\label{froholdt2}
|\beta_j \wedge w| > m \mbox{ for some }w \in E.
\eeq

Suppose first that $m > |\delta_i|$. Then we have a reduced product $v = \delta_i v'$ for some $v' \in F_n$ and $\delta_jv'$ is also reduced. Since $|\delta_i| < |\delta_j|$, we get 
$$|\beta_j \wedge \delta_jv'| = |\delta_j\beta' \wedge \delta_jv'| = |\delta_i\beta' \wedge \delta_iv'| + |\delta_j|-|\delta_i| = m + |\delta_j|-|\delta_i| > m.$$
Since $\delta_jv' = \eta\inv \delta_iv' = \eta\inv v \in E$, (\ref{froholdt2}) holds in this case.

Thus we may assume that $m \leq |\delta_i|$. Let $x$ be the prefix of length $m+1$ of $\delta_j$. Write $\eta = tct\inv$ as a reduced product with $c$ cyclically reduced. Since $|\delta_i| < |\delta_j|$, we have $c \neq 1$. Let $r \geq 1$ be such that $|c^{-r}t\inv| \geq |v|$. Then
$$\oo{\eta^{-r-1}v} = \oo{tc^{-r-1}t\inv v} = tc\inv \oo{c^{-r}t\inv v}$$
and so $tc\inv$ is a prefix of $\eta^{-r-1}v \in E$. Now 
$$|tc\inv| > \frac{|\eta|}{2} = \frac{|\delta_j\delta_i\inv|}{2} 
 > |\delta_i| \geq m$$
 since $|\delta_j| > 3|\delta_i|$. The latter also implies that $|\delta_j \wedge \eta\inv| > |\delta_i| \geq m$, hence $x$ is a prefix of $tc\inv$. Thus $x$ is a prefix of both $\beta_j$ and $\eta^{-r-1}v$. 

Hence (\ref{froholdt2}) holds for $w = \eta^{-r-1}v$. Therefore (\ref{froholdt1}) holds and so 
the action of $E$ on the regular points of $\Eq(\wh\phi,\wh\psi)$ has finitely many orbits.
\qed

\bc
Let $\phi,\psi$ be injective endomorphisms of $F_n$. Then the PCP answers \texttt{YES} on input $\phi,\psi$, i.e., $\Eq(\phi,\psi)$ is trivial  if and only if $\Eq(\wh\phi,\wh\psi)$ is finite.
\ec

Scott conjecture asked whether $\rk(\Fix(\phi))\leq n$ if $\phi\in \Aut(F_n)$ and was finally settled in 1992 in \cite{[BH92]}.
In \cite{[Coo87]}, Cooper posed the infinite version of this question: is there a bound on the number of $\Fix(\phi)$-orbits of $\Fix(\wh\phi)$? Equivalentely, is there a bound on the number of infinite fixed points needed to, together with $\Fix(\phi)$, generate $\Fix(\wh\phi)$?

Similarly, the equalizer conjecture asks if there is a bound on the number of generators on the equalizer depending on the rank of the group. The case of  the free group of rank 2 was solved in \cite{[Log22]}, but the general case is still open. We pose here the infinite version of this conjecture.

\begin{conjecture}\label{infinite stallings}
Let $F_n$ be a finitely generated free group. There is an upper bound, depending only on $n$, for the number of $\Eq(\phi,\psi)$-orbits of the regular points of $\Eq(\wh\phi,\wh\psi)$.
\end{conjecture} 

\section{Fixed points at infinity}

In this section, we will tackle the question of deciding whether the continuous extension of an injective endomorphism of the free group has nontrivial fixed points or not. We will start by assuming that our endomorphism is bijective.

\subsection{Automorphisms}
\label{sec: prelim}
We will now present the basic definitions on rational subsets of groups. For more detail, the reader is referred to \cite{[Ber79]} and \cite{[BS21]}. Let A be a finite set, $\wt{A} = A \cup A\inv$, $G = \langle A\rangle$ be a (finitely generated) group, and $\pi:\tilde A^*\to G$ be the canonical (surjective) homomorphism. 

We say that $L\subseteq \tilde A^*$ is a rational language if $L$ can be obtained from finite subsets of $\tilde A^*$ using finitely many times the operators union, product and star, where $X^*=\bigcup_{n\geq 0}X^n.$ Equivalently, rational languages are those recognized by finite automata.
A subset $K\subseteq G$ is said to be \emph{rational} if there is some rational language $L\subseteq \tilde A^*$ such that $L\pi=K$.
We will denote by $\Rat(G)$ the class of rational  subsets of $G$. Rational subsets generalize the notion of finitely generated subgroups:

\bt[\cite{[Ber79]}, Theorem III.2.7]
\label{AnisimovSeifert}
Let $H$ be a subgroup of a group $G$. Then $H\in \Rat(G)$ if and only if $H$ is finitely generated.
\et

In case the group $G$ is a free group (which will be our case) with basis $A$ with surjective homomorphism $\pi:\tilde A^*\to G$, Benois' Theorem provides us with a useful characterization of rational subsets in terms of reduced words representing the elements in the subset.
\bt[\cite{[Ben79]}]
Let $F$ be a finitely generated free group with basis $A$. Then, a subset of $F$ is rational if and only if the language of reduced words representing its elements is rational. \et

It is an easy consequence of Benois's Theorem that rational subsets of the group are effectively closed under union, intersection,  and complement, because we can simply perform the operations on the rational language of reduced words representing the elements in the subsets.\\

Fixed points of endomorphisms of free groups are known to  constitute a finitely generated and computable subgroup \cite{[Mut22]}. We start by showing that this implies that we can compute the set of elements that are mapped to themselves multiplied by a fixed element $u$.

\bt
\label{eqphi}
Let $\phi \in {\rm End}(F_n)$ and let $u \in F_n$. Let $X_{\phi,u} = \{ v \in F_n \mid v\phi = vu\}$. Then
$X_{\phi,u} \in {\rm Rat}(F_n)$ and is effectively computable.
\et

\proof
Let $A$ be a basis of $F_n$ and let $B = A \,\dot{\cup}\, \{ \$ \}$. Let $\Phi \in {\rm End}(F_B)$ extend $\phi$ through $\$\Phi = u\inv\$$. We claim that 
\beq
\label{eqphi1}
X_{\phi,u} = ({\rm Fix}(\Phi) \cap F_n\$)\$\inv.
\eeq

Indeed, let $v \in X_{\phi,u}$. Then $v\$ \in F_n\$$ and $(v\$)\Phi = (v\phi)u\inv\$ = v\$$, hence $v = v\$\$\inv \in ({\rm Fix}(\Phi) \cap F_n\$)\$\inv$.

Conversely, let $w \in {\rm Fix}(\Phi) \cap F_n\$$. Then $w = v\$$ for some $v \in F_n$ and we get $v\$ = w\Phi = (v\phi)u\inv\$$. Thus $v =  (v\phi)u\inv$ and so $v\phi = vu$. Therefore $w\$\inv = v \in X_{\phi,u}$ and so (\ref{eqphi1}) holds.

By Mutanguha's result \cite{[Mut22]}, ${\rm Fix}(\Phi)$ is an effectively computable finitely generated subgroup of $F_B$, hence rational. Now the theorem follows from (\ref{eqphi1}) and well-known closure properties of rational subsets of free groups.
\qed

We can now prove the main result of this subsection.

\bt
\label{trivialfix}
It is decidable, for a given $\phi \in {\rm Aut}(F_n)$, whether or nor ${\rm Fix}(\wh{\phi})$ is trivial.
\et

\proof
We may assume that ${\rm Fix}(\phi)$ is trivial. Let $M$ be a bounded reduction constant for $\phi$ (i.e. if the product $uv$ is reduced in $F_n$, then $|(uv)\phi| \geq |u\phi| + |v\phi| -2M$) and $\phi\inv$. Let
$$N = \max \{ |u\phi\inv| : |u| \leq M\}.$$
By Theorem \ref{eqphi}, we may decide whether or not 
$$X =  \bigcup_{|u| \leq 2M+N} X_{\phi,u} \cup X_{\phi\inv,u}$$
is infinite. If $X$ is infinite, it follows easily from compactness of $\wh{F_n}$ that $X$ has an accumulation point which is a nontrivial fixed point of $\wh{\phi}$ or $\wh{\phi\inv}$ (but they have the same fixed points!). Thus we may assume that $X$ is finite and we can therefore compute all its elements. Let $m = \max\{ |x| : x \in X\}$,
$$K = \max \{ |y\phi\inv| : |y| \leq m+M\}$$
and $s = \max\{ m+1,K+1\}$. We claim that
\beq
\label{trivialfix1}
{\rm Fix}(\wh{\phi}) \neq \{ 1\} \iff \exists v \in F_n(m < |v| \leq s
\mbox{ and $v$ is a prefix of either $v\phi$ or $v\phi\inv$}).
\eeq

Indeed, suppose that $v$ is a prefix of $v\phi$ of length $> m$ 
and write $v\phi = vu$. Since $|v| > m$, we must have  $|u| > 2M$. Let $a$ denote the first letter of $u$. It follows from our choice of $M$ that $va$ is a prefix of $(va)\phi$. Iterating the process, we construct an infinite word $\alpha \in \wh{F_n}$ such that $v\alpha^{[k]}$ is a prefix of $(v\alpha^{[k]})\phi$ for every $k$. It follows that  $v\alpha \in {\rm Fix}(\wh{\phi})$ and is nontrivial.

If $v$ is a prefix of $v\phi\inv$ of length $> m$, 
then we obtain a nontrivial fixed point of $\wh{\phi\inv}$, which is also a fixed point of $\wh{\phi}$.

Conversely, assume that $\alpha$ is a nontrivial fixed point of $\wh{\phi}$. By our first assumption, $\alpha$ must be an infinite word. Let $v = \alpha^{[s]}$ 
and write $\alpha = v\beta$. Then $\alpha = (v\phi)(\beta\wh{\phi})$. Since $M$ is a bounded reduction constant for $\phi$, there are at most $M$ letters of $v\phi$ erased in the reduction of $(v\phi)(\beta\wh{\phi})$. Hence there is some word $u$ of length $\leq M$ such that $v\phi = wu$, $\beta\wh{\phi} = u\inv\beta'$ and $\alpha = w\beta'$ are all reduced. 

Now $v$ and $w$ are both prefixes of $\alpha$. We split our argument into three cases:

\smallskip

\noindent
\underline{Case 1}: $\big{|} |v| -|w|\big{|} \leq M+N$. 

\smallskip

\noindent
Then $w = vz$ for some $z \in F_n$ of length $\leq M+N$. Hence $v\phi = wu = vzu$. Since $|zu| \leq |z| + |u| \leq M+N+M$, we get $v \in X$, contradicting $|v| > m$. Therefore this case cannot occur at all.

\smallskip

\noindent
\underline{Case 2}: $|w| > |v| + M+N$.

\smallskip

\noindent
Then $v\phi = wu = vzu$ is a reduced product for some $z \in F_n$ and so $v$ is a prefix of $v\phi$.

\smallskip

\noindent
\underline{Case 3}: $|v| > |w| + M+N$.

\smallskip

\noindent
Then we have a reduced product $v = wz$ for some $z \in F_n$ of length $> M+N$. Since $v\phi = wu$, we get $v = (w\phi\inv)(u\phi\inv)$. Now $|u| \leq M$ implies $|u\phi\inv| \leq N$, hence $\big{|} |v| -|w\phi\inv|\big{|} \leq N$. Since $|v| > |w| + M+N$, it follows that $w$ must be a prefix of $w\phi\inv$. Since $|w| < |v|$, we have $|w| < s$. Thus it suffices to show that $|w| > m$. First, we note that $|v| = s > K$ implies that $|v\phi| > m+M$. Now $v\phi = wu$ and $|u| \leq M$ together imply $|w| \geq |v\phi|-M > m$ and so (\ref{trivialfix1}) holds. 

Since condition (\ref{trivialfix1}) is clearly decidable, the theorem is proved.
\qed

\subsection{Almost length-increasing monomorphisms}

Let $R_n$ denote the set of all reduced words on the alphabet $\wt{A}$.
We now focus on proving the result for almost length-increasing monomorphisms.
An endomorphism $\p$ of $F_n$ is said to be:
\bi
\item
{\em length-increasing} if $|u\p| \geq |u|$ for each $u \in F_n$;
\item
{\em strictly length-increasing} if $|u\p| > |u|$ for each $u \in F_n\setminus \{1\}$;
\item
{\em almost length-increasing} if $|u\p| < |u|$ for only finitely many $u \in F_n$;
\item
{\em almost strictly length-increasing} if $|u\p| \leq |u|$ for only finitely many $u \in F_n$.
\ei

It is clear that strictly length-increasing implies both length-increasing and almost strictly length-increasing, and any of the latter two properties implies almost length-increasing.  It is also easy to see that almost length-increasing endomorphisms are injective. Indeed, if $\phi$ is noninjective, the kernel must be infinite, and for every nontrivial  element $x$ in the kernel, we have that $|x\phi|<|x|$.   For showing that no other implication holds in general, it suffices to show that:
\bi
\item there exists some length-increasing endomorphism of $F_n$ which is not almost strictly length-increasing;
\item there exists some almost strictly length-increasing endomorphism of $F_n$ which is not length-increasing.
\ei
In fact, the identity homomorphism is length-increasing but not almost strictly length-increasing. The following example serves as a counterexample for the second claim:

\be
Let $\p$ be the endomorphism of $F_2$ defined by $a\p =a^2$ and $b\p = a^2b^2a^{-2}$. Then $\p$ is an almost strictly length-increasing monomorphism of $F_2$ which is not length-increasing.
\ee

Indeed, let $\psi$ be the  inner automorphism of $F_2$ defined by $u\psi = a^{-2}ua^2$ and let $\theta = \p\psi$. Since $a\theta = a^2$ and $b\theta = b^2$, then $\theta$ is a monomorphism and so must be $\p$. Moreover, $|u\theta| = 2|u|$ for every $u \in F_n \setminus \{ 1\}$, hence 
$$|u\p| = |u\theta\psi\inv| = |a^2(u\theta)a^{-2}| \geq |u\theta| -4 = 2|u|-4$$
and so 
$$|u| > 4 \Rw |u\p| \geq 2|u| - 4 > |u|.$$
Thus $\p$ is almost strictly length-increasing. Since $(a^{-1}ba)\p = b^2$, $\p$ is not length-increasing.

\medskip

The following remark shows that these concepts are not at all interesting for automorphisms (we get at most letter permutations):

\brm
Let $\p$ be an automorphism of $F_n$. Then the following conditions are equivalent:
\bi
\item[(i)] $\p$ is almost length-increasing;
\item[(ii)] $|a\p| = 1$ for every $a \in A$.
\ei 
\erm

\proof
(i) $\Rw$ (ii). Suppose that $|a\p\inv| \neq 1$ for some $a \in A$. Then $a\p\inv = u$ for some $u \in F_n$ with $|u| > 1$. Write $u = vcv\inv$ in reduced form with $c$ cyclically reduced. Since $|u| > 1$, we have $|vc| > 1$. For every $k \geq 1$, we get 
$$(vc^kv\inv)\p = (vcv\inv)^k\p = u^k\p = a^k.$$
Since $|vc| > 1$, we have $|(vc^kv\inv)\p| = |a^k| = k < |vc^kv\inv|$ for every $k \geq 1$, hence $\p$ is not almost length-increasing.

(ii) $\Rw$ (i). Immediate. 
\qed\\

We devote henceforth our attention to monomorphisms.  We will start by showing that these conditions are decidable, that is, given an endomorphism of a free group, it is decidable whether it is injective, and, in case it is, we can decide if it is length-increasing, strictly length-increasing, almost length-increasing or almost strictly length-increasing.  But first we recall the notion of Stallings automaton. For details, see \cite{[BS21]}. 

Given a finitely generated subgroup $H$ of $F_n$. Take $S$ to be a finite generating set of $H$. The {\em flower automaton} of $S$ is obtained by gluing to a basepoint petals labelled by the reduced forms of the elements of $S$. Then we perform all the possible foldings of pairs of edges of the form
$$\xymatrix{
\bullet & \bullet \ar[l]_a \ar[r]^a& \bullet
}$$
where $a$ is a letter (or its inverse). We obtain a finite inverse automaton which does not depend on either $S$ or the sequence of foldings and is therefore known as the {\em Stallings automaton} of $H$ and denoted by $\st(H)$. We write $\st(H) = (Q,q_0,E)$ when $Q$ is the vertex set, $q_0 \in Q$ the basepoint and $E$ the edge set.

\bl
\label{decmono}
Let $\p$ be an endomorphism of $F_n$ and let  $\st({\rm Im}(\p)) = (Q,q_0,E)$. Then the following conditions are equivalent:
\bi
\item[(i)] $\p$ is a monomorphism;
\item[(ii)] $\frac{|E|}{2} = |Q|+n-1$.
\ei
\el

\proof
Since $F_n$ is hopfian, then $\p$ is a monomorphism if and only if ${\rm Im}(\p)$ has rank $n$.  By a well-known fact about Stallings automata, we can compute a basis of ${\rm Im}(\p)$ by taking a spanning tree $T$ of $\A$ and building a generator for each positive edge of $\A$ not occurring in $T$. The number of positive edges in $T$ is $|Q|-1$, so the rank of ${\rm Im}(\p)$ is $\frac{|E|}{2} - (|Q|-1)$. Therefore ${\rm Im}(\p)$ has rank $n$ if and only if $\frac{|E|}{2} = |Q|+n-1$.
\qed

\bp
\label{decli}
It is decidable whether or not a given monomorphism $\p$ of $F_n$ is:
\bi
\item[(i)]
length-increasing;
\item[(ii)]
strictly length-increasing.
\ei
\ep

\proof
Let $M$ denote a bounded reduction constant for $\p$ (which is computable by \cite{[BFH97]}, see also \cite[Theorem 7.1]{[Sil13]}). For every $u \in R_n$ let $u\lambda$ be the the last letter of $u$ if $u$ is nonempty, and $1$ otherwise. Write also
$$u\alpha = \wedge\{ (uv)\p \mid uv \in R_n\},\quad u\p = (u\alpha)(u\beta),\quad q_u = (u\lambda, u\beta, |u\p|-|u|).$$
By our choice of $M$, we have 
\beq
\label{decli4}
\mbox{$|u\beta| \leq M$ for every $u \in R_n$.}
\eeq

Let $u,v \in R_n$ and $a \in \wt{A}$. We show that:
\beq
\label{decli1}
\mbox{if $q_u = q_v$, then $ua \in R_n$ if and only if $va \in R_n$ and in that case $q_{ua} = q_{va}$.}
\eeq

Indeed, suppose that $q_u = q_v$. Then $u\lambda = v\lambda$ and so $ua \in R_n$ if and only if $va \in R_n$. Assume that this is the case. Then $(ua)\lambda  = a = (va)\lambda$. Write $w = u\beta = v\beta$. For every $z \in R_n$, we have $uaz \in R_n$ if and only if $vaz \in R_n$. Moreover, it follows from the definition of $\alpha$ and $\beta$ that $\oo{(uaz)\p} = (u\alpha)\oo{w(az)\p}$ and $\oo{(vaz)\p} = (v\alpha)\oo{w(az)\p}$ are reduced products, hence 
\beq
\label{decli2}
(ua)\alpha = (u\alpha)(\wedge\{ \oo{w(az)\p} \mid az \in R_n\})\quad \mbox{and} \quad (va)\alpha = (v\alpha)(\wedge\{ \oo{w(az)\p} \mid az \in R_n\}),
\eeq
yielding $(ua)\beta = (va)\beta$. Now
$$|(ua)\p| -|ua| = |(ua)\alpha| + |(ua)\beta| -|u\alpha| + |u\alpha| - |u| -1$$
and
$$|(va)\p| -|va| = |(va)\alpha| + |(va)\beta| -|v\alpha| + |v\alpha| - |v| -1.$$
Since $(ua)\beta = (va)\beta$, $q_u = q_v$ implies 
$$|u\alpha| - |u| = -|u\beta| + |u\p| -|u|= -|v\beta| + |v\p| -|v| = |v\alpha| - |v|$$
and (\ref{decli2}) implies $|(ua)\alpha| -|u\alpha| = |(va)\alpha| -|v\alpha|$, we get $|(ua)\p| -|ua| = |(va)\p| -|va|$. Thus $q_{ua} = q_{va}$ and so (\ref{decli1}) holds.

Let $\A = (Q,q_1,Q,E)$ be the $\wt{A}$-automaton defined by:
\bi
\item
$Q = \{ q_u \mid u \in R_n\}$;
\item
$E = \{ (q_u,a,q_{ua}) \mid a \in \wt{A};\, u,ua \in R_n \}.$
\ei  

We show that $\p$ is length-increasing if and only if
\beq
\label{decli3}
\mbox{there is no $u \in R_n$ of length $\leq 2M(2n+1)^{M+1}$ such that $|u\p| < |u|$.}
\eeq

The direct implication holds trivially, so we assume that $\p$ is not length-increasing. Let $u \in R_n$ have minimum length among all the words having shorter image under $\p$. Consider the path $q_1 \xr{u} q_u$ in $\A$. Suppose that some vertex is repeated in this path. Then we can find some reduced factorization $u = u_1u_2u_3$ such that $q_{u_1} = q_{u_1u_2}$ and $u_2 \neq 1$. By successively applying (\ref{decli1}), we get that $u_1u_3 \in R_n$ and $q_{u_1u_3} = q_{u_1u_2u_3} = q_u$. Hence $|(u_1u_3)\p| -|u_1u_3| = |u\p| - |u| < 0$, contradicting the minimality of $|u|$. Thus no vertex is repeated in the path $q_1 \xr{u} q_u$. 

Suppose now that the third component of some vertex  in this path is $\geq 2M$. Then there is some reduced factorization $u = vw$ such that $|v\p| - |v| \geq 2M$. Then
$$|v| + |w| = |u| > |u\p| = |(vw)\p| \geq |v\p| + |w\p| -2M \geq |v| + |w\p|$$
and so $|w\p| < |w|$. Since $v \neq 1$, this contradicts once again the minimality of $|u|$. Thus the third component of every vertex in the path $q_1 \xr{u} q_u$ is $< 2M$ and by minimality of $|u|$ this component is $< 0$ only at the last vertex. Thus any other vertex is of the form $q_v = (v\lambda, v\beta, |v\p|-|v|)$ with $0 \leq |v\p|-|v| \leq 2M-1$. Now there are $2n+1$ possible choices for $v\lambda$, less than $(2n+1)^M$ possible choices for $v\beta$ (by (\ref{decli4})) and $2M$ possible choices for $|v\p|-|v|$. Thus there are at most $2M(2n+1)^{M+1}$ distinct vertices in $q_1 \xr{u} q_u$. Since no vertex is repeated in the path, then $|u|$ equals the number of vertices appearing in the path before the last one. Therefore $|u| \leq 2M(2n+1)^{M+1}$ and so $\p$ is length-increasing if and only if (\ref{decli3}) holds. Since the latter condition is obviously decidable, we may decide whether or not $\p$ is length-increasing.

We may obviously adapt the argument above to prove that
$\p$ is strictly  length-increasing if and only if
there is no $u \in R_n \setminus \{ 1\}$ of length $\leq 2M(2n+1)^{M+1}$ such that $|u\p| \leq |u|$. Therefore we may decide whether or not $\p$ is strictly length-increasing.
\qed

\bp
\label{combou}
Let $\p$ be a length-increasing monomorphism of $F_n$ and let $k \geq 0$. Then 
$$L_k = \{ u \in R_n \; \big{|}\; |u\p| \leq |u| + k\}$$
is rational and computable.
\ep

\proof
Let $M$ denote a bounded reduction constant for $\p$. Recall the $\wt{A}$-automaton $\A = (Q,q_1,Q,E)$ from the proof of Proposition \ref{decli}. Suppose that there exists a path $p \xr{a} q \xr{a\inv} r$ in $\A$ for some $a \in \wt{A}$.
Then there exist $u,v \in R_n$ such that $p = q_u$, $q = q_{ua} = q_v$, $r = q_{va\inv}$ and $ua,va\inv \in R_n$. Then $v\lambda = (ua)\lambda = a$, contradicting $va\inv \in R_n$. Thus there are no paths of the form $p \xr{a} q \xr{a\inv} r$ in $\A$ and so $L(\A) \subseteq R_n$.

We claim also that $\A$ is deterministic.
Indeed, let $(p,a,q),(p,a,r) \in E$. Then there exist $u,v \in R_n$ such that $p = q_u = q_v$, $q = q_{ua}$ and $r = q_{va}$. It follows from (\ref{decli1}) that $q = q_{ua} = q_{va} = r$, therefore $\A$ is deterministic.

Let 
$$Y_k = \{ q_u \in Q \; \big{|}\; |u\p| \leq |u| + k + 2M\}.$$
We have $q_u = (u\lambda, u\beta, |u\p|-|u|)$, $u\lambda \in \wt{A} \cup \{ 1\}$ and $|u\beta| \leq M$ by (\ref{decli4}). Since $\p$ is length-increasing, it follows that $Y_k$ is finite. 

We define a sequence $(\A_{k,j})_j$ of finite subautomata $\A_{k,j} = (Q_{k,j},q_1,T_{k,j},E_{k,j})$ of $\A$ with $Q_{k,j} \subseteq Y_k$ as follows:
We set $Q_{k,1} = T_{k,1} = \{ q_1\}$ and $E_{k,1} = \emptyset$. 
If $\A_{k,j}$ is defined for some $j \geq 1$, we define
\bi
\item
$Q_{k,j+1} = Q_{k,j} \cup \{ q_{ua} \mid q_u\in Q_{k,j},\, |(ua)\p| \leq |ua| + k+2M,\, a \in \wt{A};\, u,ua \in R_n \}$;
\item
$T_{k,j+1} = T_{k,j} \cup \{ q_{u} \in Q_{k,j+1} \; \big{|}\; |u\p| \leq |u| + k\}$;
\item
$E_{k,j+1} = E_{k,j} \cup \{ (q_u,a,q_{ua}) \mid q_u\in Q_{k,j},\, |(ua)\p| \leq |ua| + k+2M,\, a \in \wt{A};\, u,ua \in R_n \}$.
\ei
Since $\A$ is deterministic and $Y_k$ is finite, it follows easily by induction on $j$ that each $\A_{k,j}$ is finite and computable. Since $\A_{k,j}$ is always a subautomaton of $\A_{k,j+1}$, there exists some $m \geq 1$ such that $\A_{k,m} = \A_{k,m+1}$. Hence $\A_{k,m} = \A_{k,m+j}$ for every $j \geq 1$. Setting $\A_{k} = \A_{k,m}$, we have a finite computable deterministic subautomaton of $\A$. 
To complete our proof, it suffices to show that $L_k = L(\A_k)$.

Suppose first that $u \in L(\A_k)$. Since $\A_k$ is a subautomaton of $\A$ and $L(\A) \subseteq R_n$, then $u \in R_n$. Since $\A$ is deterministic, there is in $\A_k$ a unique path $q_1 \xr{u} q_u \in T_k$. Hence $q_u = q_v$ for some $v \in R_n$ satisfying $|v\p| \leq |v| + k$. Since $q_u = q_v$ implies $|u\p| - |u| = |v\p| - |v| \leq k$, we get $u \in L_k$. Therefore $L(\A_k) \subseteq L_k$.

Suppose now that $u \in L_k$. Let $v$ be a prefix of $u$ and write $u = vw$. Suppose that $|v\p| - |v| > k+2M$. By our choice of $M$, and since $\p$ is length-increasing, we get
$$|u\p| \geq |v\p| + |w\p| - 2M > |v| + k+2M + |w| -2M = |u| + k,$$
contradicting $u \in L_k$. Hence every prefix $v$ of $u$ satisfies $|v\p| - |v| \leq k+2M$. Write $u = a_1\ldots a_r$ in reduced form with $a_1,\ldots, a_r \in \wt{A}$. Now it follows easily that
$$q_{1} \xr{a_1} q_{a_1} \xr{a_2} q_{a_1a_2} \xr{a_3} \ldots \xr{a_r} q_{a_1\ldots a_r} = q_u \in T_k$$
is a path in $\A_k$, thus $u \in L(\A_k)$ and so $L_k \subseteq L(\A_k)$. Therefore $L_k = L(\A_k)$ and we are done.
\qed

\bp
\label{decali}
It is decidable whether or not a given monomorphism $\p$ of $F_n$ is:
\bi
\item[(i)]
almost length-increasing;
\item[(ii)]
almost strictly length-increasing.
\ei
In the affirmative case, we can compute all the words which are exceptions to the corresponding condition. 
\ep

\proof
Let $K = \max\{ |a\p| : a \in \wt{A} \}$.
We show first that
\beq
\label{decali1}
\mbox{if $\p$ is almost length-increasing, then $|u\p| \geq |u|-K$ for every $u \in R_n$.}
\eeq

Indeed, suppose that $|u\p| < |u|-K$ for some $u \in R_n$. Take $a \in \wt{A}$ such that $ua$ is cyclically reduced. Then
$$|(ua)\p| \leq |u\p| + |a\p| < |u|-K+K = |u|.$$
For each $m \geq 1$, $(ua)^m \in R_n$ and 
$$|(ua)^m\p| \leq m|(ua)\p| < m|u| < |(ua)^m|,$$
contradicting $\p$ being almost length-increasing. Thus (\ref{decali1}) holds.

Now let $b,c$ denote two new letters and take $B = A \cup \{b,c\}$ as a basis for $F_{n+2}$. Let $\psi \in F_{n+2}$ be defined by
$$b\psi = b\inv cb, \quad c\psi = c,\quad a\psi = b^K(a\p)b^{-K}\quad (a \in A).$$
Let $\A = (Q,q_0,E)$ be the Stallings automaton of ${\rm Im}(\p)$. By Lemma \ref{decmono}, we have
$\frac{|E|}{2} = |Q|+n-1$. Now let $\A' = (Q',q'_0,E')$ be the Stallings automaton of ${\rm Im}(\psi)$. It is easy to see that $\A'$ is of the form
$$
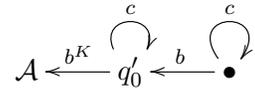
\xymatrix{
\A & q'_0 \ar[l]_{b^K} \ar@(ul,ur)^{c} & \bullet \ar[l]_{b} \ar@(ul,ur)^{c}
}$$
Hence $|Q'| = |Q| + K+1$ and $|E'| = |E| + 2(K+3)$. It follows that 
$$\frac{|E'|}{2} = \frac{|E|}{2} +K+3 =|Q|+n-1 + K+3 = |Q'| + (n+2)-1$$
and so $\psi$ is a monomorphism by Lemma \ref{decmono}. Let $M$ denote a bounded reduction constant for $\psi$. 

By Proposition \ref{decli}, we can decide whether or not $\psi$ is length-increasing. Suppose that $\psi$ is not length-increasing. Then there exists some $u \in R_{n+2}$ such that  $|u\psi| < |u|$. Write $u = v_0w_1v_1 \ldots w_mv_m$ with $m \geq 0$, $v_0,v_m$ reduced words on $\wt{\{b,c\}}$, $v_1,\ldots,v_{m-1}$ nonempty reduced words on $\wt{\{b,c\}}$, $w_1,\ldots,w_m \in R_n \setminus \{ 1\}$. It follows from the definition of $\psi$ that the product
$$\oo{u\psi} = \oo{v_0\psi}\,\oo{w_1\psi}\,\oo{v_1\psi} \ldots \oo{w_m\psi}\,\oo{v_m\psi}$$
is reduced and $|v_i\psi| \geq |v_i|$ for every $i$. Hence $|w_j\psi| < |w_j|$ for some $w_j \in R_n$. Since $w_j\psi = b^K(w_j\p)b^{-K}$, we have $|w_j\p| = |w_j\psi| -2K < |w_j| -2K$ and by (\ref{decali1}) $\p$ is not almost length-increasing.

Thus we may assume that $\psi$ is length-increasing. Let
$$L = \{ w \in R_n \; \big{|} \; |w\p| < |w| \}.$$
 For every $w \in R_n$, we have
$|w\p| - |w| = |w\psi| - |w| - 2K \geq -2K$, hence
$$L = \{ w \in R_n \mid -2K \leq |w\p| - |w| < 0 \}.$$

Consider the computable finite $\wt{B}$-automaton $\A_{2K-1}$
from the proof of Proposition \ref{combou} (built with respect to $\psi$). We have
$$L(\A_{2K-1}) = \{ u \in R_{n+2} \; \big{|}\; 0 \leq |u\psi| - |u| \leq 2K-1\}.$$
Since $|w\p| - |w| = |w\psi| - |w| - 2K$ for every $w \in R_n$, we get
$$L(\A_{2K-1}) \cap R_n = \{ w \in R_{n} \; \big{|}\; 0 \leq |w\psi| - |w| \leq 2K-1\}
= \{ w \in R_n \mid -2K \leq |w\p| - |w| < 0 \} = L,$$
hence $L$ is an effectively computable rational language. Thus 
we can decide whether or not $L$ is finite, that is, 
whether or not $\p$ is almost length-increasing. And we can compute all the finitely many exceptions.

For deciding whether or not $\p$ is almost strictly length-increasing, we consider 
$$L = \{ w \in R_n \; \big{|} \; |w\p| \leq |w| \}$$
and perform the obvious adaptations.
\qed

\bl
\label{tau}
Let $\p$ be a monomorphism of $F_n$ such that ${\rm Fix}(\p) = \{1\}$ and let $v \in F_n$. Then there is at most one $u \in F_n$ such that $u\p = uv$, and in that case it is computable.
\el

\proof
Suppose that $u\p = uv$ and $w\p = wv$. Then
$(uw\inv)\p = uw\inv$ and ${\rm Fix}(\p) = \{1\}$ yields $u = w$. Computability follows from Theorem \ref{eqphi}.
\qed\\ 

We say that $\alpha \in \partial F_n$ is {\em computable} if every finite prefix of $\alpha$ is computable. A finite subset $X$ of $\partial F_n$ is computable if we can enumerate its elements, say $X = \{ \alpha_1,\ldots,\alpha_m\}$, and each $\alpha_i$ is computable.  We can now prove the main result of this subsection, that is, that the infinite fixed points of an almost length-increasing monomorphism (without finite fixed points) are computable, yielding a proof that it is decidable whether an almost length-increasing monomorphism has infinite fixed points or not. 

\bt
\label{alifix}
Let $\p$ be an almost length-increasing monomorphism of $F_n$ with ${\rm Fix}(\p) = \{1\}$. Then ${\rm Fix}(\wh{\p})$ is computable.
\et

\proof
Since ${\rm Fix}(\p) = \{1\}$, there exist no infinite singular fixed points and it follows from \cite[Theorem 8.5]{[Sil13]} that ${\rm Fix}(\wh{\p})$ is finite.

Given $u \in R_n$, write
$$u\sigma = (u \wedge \oo{u\p}),\quad u = (u\sigma)(u\tau), \quad \oo{u\p} = (u\sigma)(u\rho).$$

Let $M$ denote a bounded reduction constant for $\p$. 
Let
$$X_1 = \{ u \in R_n \mid u\tau = 1\},\quad X_2 = \{ u \in R_n \; \big{|}\; |u\rho| \leq M\}, \quad X = X_1 \cup X_2.$$
 Given $\alpha \in \partial F_n$, let $\pref(\alpha)$ denote the set of all (finite) prefixes of $\alpha$.
We claim that
\beq
\label{alifix1}
\fix(\wh{\p}) = \{ 1\} \cup \{ \alpha \in \partial F_n \mid \pref(\alpha) \subseteq X \}.
\eeq

Indeed, let $\alpha = a_1a_2\ldots \in \partial F_n$ with $a_1,a_2,\ldots \in \wt{A}$. For each $k \geq 0$,  
write $\alpha = \alpha^{[k]}\beta_k$. 

Suppose that $\alpha^{[k]}\tau \neq 1$ and $|\alpha^{[k]}\rho| > M$ for some $k$. We have
$$\alpha\wh{\p} = (\alpha^{[k]}\beta_k)\wh{\p} = \oo{(\alpha^{[k]}\p)(\beta_k\wh{\p})}.$$
Since this reduction erases at most $M$ letters from $\alpha^{[k]}\p$ and $|\alpha^{[k]}\rho| > M$, 
we get $(\alpha \wedge \alpha\wh{\p}) = (\alpha^{[k]} \wedge \alpha^{[k]}\p)$, hence $\alpha \notin \fix(\wh{\p})$.

Conversely, suppose that $\alpha \notin \fix(\wh{\p})$. Then $(\alpha \wedge \alpha\wh{\p}) = \alpha^{[r]}$ for some $r \geq 0$. Since $\alpha = \lim_{k\to +\infty} \alpha^{[k]}$, then by continuity we get $\alpha\wh{\p} = \lim_{k\to +\infty} \alpha^{[k]}\p$. Hence there exists some $m > r$ such that $|\alpha^{[m]}\p \wedge \alpha\wh{\p}| > r+M$. Then $(\alpha^{[m]} \wedge \alpha^{[m]}\p) = \alpha^{[r]}$ and so we have both $\alpha^{[m]}\tau \neq 1$ and $|\alpha^{[m]}\rho| > M$. Thus $\pref(\alpha) \not\subseteq X$ and so (\ref{alifix1}) holds.

We prove next that $X_2$ is finite and computable. 
Since $\p$ is an almost length-increasing monomorphism, it follows from Proposition \ref{decali} and Lemma \ref{tau} that we can compute the finite set
$$L = \{ w \in R_n \;\big{|}\; |w\p| < |w| \}.$$
Let $\ell = \max\{ |w| : w \in L\}$ and let $V$ contain all the reduced words of length $\leq 2M$. It follows from Lemma \ref{tau} that, for each $v \in V$, there is at most some $v' \in F_n$ such that $v'\p = v'v$, and in that case it is computable. Hence $V' = \{ v' \mid v \in V\}$ is finite and computable. We claim that $X_2 \subseteq L \cup V'$.

Indeed, let $u \in X_2 \setminus L$. Then $|u\tau| \leq |u\rho| \leq M$. Since $u\p = u(u\tau)\inv(u\rho)$, we get $(u\tau)\inv(u\rho) \in V$ and so $u \in V'$. Thus $X_2 \subseteq L \cup V'$ and it follows that $X_2$ is indeed finite and computable. 

Let 
$$m = \max\{ |x| : x \in X_2\} +1 \;\mbox{ and }\; Y = \{ x \in X_1 \;\big{|} \; |x| = m\}.$$
We show that:
\beq
\label{alifix3}
\mbox{for every $y \in Y$, there exists a unique $y\gamma \in \fix(\wh{\p})$ having $y$ as a prefix.}
\eeq

Given $y \in Y$, we build a sequence $(y_k)_k$ in $X_1$ where each $y_k$ is obtained by adding a letter at the end of $y_{k-1}$. Let $y_0 = y$. For $k \geq 0$, assume that $y_k \in X_1$ is defined. Since $|y_k| \geq m$, we have  $|y_k\rho| > M$. Let $a_k$ be the first letter of $y_k\rho$ and define $y_{k+1} = y_ka_k$. Since $y_{k+1}$ is a prefix of $y_k\p$, we have $y_{k+1} \in R_n$. Since $|y_k\rho| > M$ and at most $M$ letters of $y_k\rho$ are erased in the reduction of $(y_k\rho)(a_k\p)$, we get that $y_{k+1}\tau = 1$ and so $y_{k+1} \in X_1$.  By induction on $k$, the sequence $(y_k)_k$ is now defined. Let 
$$y\gamma = \lim_{k \to +\infty} y_k = ya_0a_1a_2 \ldots.$$
It follows from (\ref{alifix1}) that $y\gamma \in \fix(\wh{\p})$, and by construction it has $y$ as a prefix. It is immediate that we have no other choice for $a_k$ to reach a fixed point, so $y\gamma$ is indeed the unique fixed point having $y$ as a prefix.
 Hence (\ref{alifix3}) holds.

Finally, we show that
\beq
\label{alifix4}
\fix(\wh{\p}) = \{ y\gamma \mid y \in Y\}.
\eeq

The direct inclusion follows from (\ref{alifix3}). Conversely, let $\alpha \in \fix(\wh{\p})$ and let $y = \alpha^{[m]}$. By (\ref{alifix1}), we have $y \in X$ and it follows from the definition of $m$ that $y \in X_1$. Hence $y \in Y$ and by (\ref{alifix3}) we get $\alpha = y\gamma$. Thus (\ref{alifix4}) holds.

Since $Y$ is computable and so is $y_k$ for all $y \in Y$ and $k \geq 0$, the theorem is proved.
\qed

\bc
\label{alifixinf}
Let $\p$ be an almost length-increasing monomorphism of $F_n$. Then we can decide whether or not $\wh{\p}$ admits an infinite fixed point.
\ec

\proof
By \cite{[Mut22]}, we can decide whether or not $\fix(\p) = \{ 1\}$ and compute a nontrivial one if it exists.

Suppose that $\fix(\p) = \{ 1\}$. By the proof of Theorem \ref{alifix}, $\wh{\p}$ admits an infinite fixed point if and only if $Y \neq \emptyset$. Since $Y$ is computable, we can decide this.

Assume now that $\fix(\p) \neq \{ 1\}$. Then we can compute some nontrivial $u \in \fix(\p)$. Write $\oo{u} = vcv\inv$ with $c$ cyclically reduced. Then $(vc^kv\inv)\p = vc^kv\inv$ for every $k \geq 0$ and so
$$(vc^{\omega})\wh{\p} = \lim_{k \to +\infty} (vc^k)\p = \lim_{k \to +\infty} (vc^kv\inv)\p = \lim_{k \to +\infty} vc^kv\inv = \lim_{k \to +\infty} vc^k= vc^{\omega},$$
therefore $\wh{\p}$ admits an infinite fixed point.
\qed\\

 We remark that almost-length increasing monomorphisms include all endomorphisms with remnant, which have density $1$ among endomorphisms of the free group \cite{[Wag99]}, thus being a generic subset of the set of endomorphisms of $F_n$.

\section{Future work}  
Together with Conjecture \ref{infinite stallings}, the following are the main questions arising from this work:

\begin{itemize}
\item The main open problem is the generalization of Theorem \ref{trivialfix} to all injective endomorphisms. A more general question would be: assuming that $\Eq(\phi,\psi)$ is computable,  can we  decide triviality of $\Eq(\wh\phi,\wh\psi)$?
\item Can we answer any of the previous questions assuming that $n=2$? We remark that Logan proved that equalizers of endomorphisms of the free group of rank $2$ are computable.
\end{itemize}

\section*{Acknowledgements}
 Both authors were partially supported by
CMUP, member of LASI, which is financed by national funds through FCT - Funda\c c\~ao
para a Ci\^encia e a Tecnologia, I.P., under the project UID/00144/2025. The first author also acknowledges support by national funds through the Funda\c c\~ao
para a Ci\^encia e a Tecnologia, FCT, under the project UID/04674/2025. 
\bibliographystyle{plain}
\bibliography{Bibliografia}

\begin{thebibliography}{10}

\bibitem{[BS21]}
L.~Bartholdi and P.~V. Silva.
\newblock Rational subsets of groups.
\newblock In J.-E. Pin, editor, {\em Handbook of Automata Theory}, chapter~23.
  EMS Press, Berlin, 2021.

\bibitem{[BMS02]}
G.~Baumslag, A.~G. Myasnikov, and V.~Shpilrain.
\newblock Open problems in combinatorial group theory.
\newblock In {\em Combinatorial and geometric group theory ({N}ew {Y}ork,
  2000/{H}oboken, {NJ}, 2001)}, volume 296 of {\em Contemp. Math.}, pages
  1--38. Amer. Math. Soc., Providence, RI, 2002.

\bibitem{[Ben79]}
M.~Benois.
\newblock Parties rationnelles du groupe libre.
\newblock {\em C. R. Acad. Sci. Paris, S\'er A}, 269:1188--1190, 1969.

\bibitem{[Ber79]}
J.~Berstel.
\newblock {\em Transductions and Context-free Languages}.
\newblock Teubner, Stuttgart, 1979.

\bibitem{[BFH97]}
M.~Bestvina, M.~Feighn, and M.~Handel.
\newblock Laminations, trees, and irreducible automorphisms of free groups.
\newblock {\em Geom. Funct. Anal.}, 7(2):215--244, 1997.

\bibitem{[BH92]}
M.~Bestvina and M.~Handel.
\newblock Train tracks and automorphisms of free groups.
\newblock {\em Ann. Math.}, 135:1--51, 1992.

\bibitem{[BM16]}
O.~Bogopolski and O.~Maslakova.
\newblock An algorithm for finding a basis of the fixed point subgroup of an
  automorphism of a free group.
\newblock {\em Int. J. Algebra Comput.}, 26(1):29--67, 2016.

\bibitem{[CLL24]}
L.~Ciobanu, A.~Levine, and A.~D. Logan.
\newblock Post's correspondence problem for hyperbolic and virtually nilpotent
  groups.
\newblock {\em Bull. London Math. Soc.}, 56(1):159--175, 2014.

\bibitem{[CL21]}
L.~Ciobanu and A.~D. Logan.
\newblock Variations on the {P}ost correspondence problem for free groups.
\newblock In {\em Developments in language theory}, Lecture Notes in Computer
  Science, chapter 12811, pages 90--102. Springer, 2021.

\bibitem{[Coo87]}
D.~Cooper.
\newblock Automorphisms of free groups have finitely generated fixed point
  sets.
\newblock {\em J. Algebra}, 111:453--456, 1987.

\bibitem{[DV96]}
W.~Dicks and E.~Ventura.
\newblock The group fixed by a family of injective endomorphisms of a free
  group.
\newblock {\em Contemp. Math.}, 195, 1996.

\bibitem{[DKLM19]}
V.~Diekert, O.~Kharlampovich, M.~Lohrey, and A.~Myasnikov.
\newblock Algorithmic problems in group theory.
\newblock {\em Dagstuhl seminar report 19131}, 2019.

\bibitem{[DL12]}
J.~Dong and Q.~Liu.
\newblock Undecidability of infinite {P}ost correspondence problem for
  instances of size 8.
\newblock {\em RAIRO Inform. Th\'eor. Appl.}, 46(3):451--457, 2012.

\bibitem{[Dug78]}
J.~Dugundji.
\newblock {\em Topology}.
\newblock Boston, Mass.-London-Sydney: Allyn and Bacon, Inc. Reprinting of the
  1966 original, Allyn and Bacon Series in Advanced Mathematics, 1978.

\bibitem{[Fin15]}
O.~Finkel.
\newblock The exact complexity of the infinite {P}ost correspondence problem.
\newblock {\em Inform. Process. Lett.}, 115(6-8):609--611, 2015.

\bibitem{[Ger87]}
S.~M. Gersten.
\newblock Fixed points of automorphisms of free groups.
\newblock {\em Adv. Math.}, 64:51--85, 1987.

\bibitem{[GH90]}
E.~Ghys and P.~de~la Harpe.
\newblock {\em Sur les Groupes Hyperboliques d'apr\`es {M}ikhail {G}romov}.
\newblock Birkh\"auser, Boston, 1990.

\bibitem{[GT89]}
R.~Z. Goldstein and E.~C. Turner.
\newblock Fixed subgroups of homomorphisms of free groups.
\newblock {\em Bull. London Math. Soc.}, 18(5):468--470, 1989.

\bibitem{[IT89]}
W.~Imrich and E.~Turner.
\newblock Endomorphisms of free groups and their fixed points.
\newblock {\em Math. Proc. Cambridge Philos. Soc.}, 105(3):421--422, 1989.

\bibitem{[Log22]}
A.~D. Logan.
\newblock The equalizer conjecture for the free group of rank two.
\newblock {\em Q. J. Math}, 73(2):777--793, 2022.

\bibitem{[MS21]}
F.~Matucci and P.~V. Silva.
\newblock Extensions of automorphisms of self-similar groups.
\newblock {\em J. Group Theory}, 24:857--897, 2021.

\bibitem{[Mut22]}
J.~P. Mutanguha.
\newblock Constructing stable images.
\newblock {\em preprint, available at
  https://mutanguha.com/pdfs/relimmalgo.pdf}, 2021.

\bibitem{[MNU14]}
A.~Myasnikov, A.~Nikolaev, and A.~Ushakov.
\newblock The {P}ost correspondence problem in groups.
\newblock {\em J. Group Theory}, 17(6):991--1008, 2014.

\bibitem{[Pau89]}
F.~Paulin.
\newblock Points fixes d'automorphismes de groupes hyperboliques.
\newblock {\em Ann. Inst. Fourier}, 39:651--662, 1989.

\bibitem{[Pos46]}
E.~L. Post.
\newblock A variant of a recursively unsolvable problem.
\newblock {\em Bull. Amer. Math. Soc.}, 52:264--268, 1946.

\bibitem{[RSS13]}
E.~Rodaro, P.~V. Silva, and M.~Sykiotis.
\newblock Fixed points of endomorphisms of graph groups.
\newblock {\em J. Group Theory}, 16(4):573--583, 2013.

\bibitem{[Sil13]}
P.~V. Silva.
\newblock Fixed points of endomorphisms of virtually free groups.
\newblock {\em Pacific J. Math.}, 263(1):207--240, 2013.

\bibitem{[Sta10]}
P.~C. Staecker.
\newblock Typical elements in free groups are in different doubly-twisted
  conjugacy classes.
\newblock {\em Topol. Appl.}, 157:1736--1741, 2010.

\bibitem{[Sta83]}
J.~Stallings.
\newblock Topology of finite graphs.
\newblock {\em Invent. Math.}, 71:551--565, 1983.

\bibitem{[Ven02]}
E.~Ventura.
\newblock Fixed subgroups in free groups: a survey.
\newblock In {\em Combinatorial and geometric group theory ({N}ew {Y}ork,
  2000/{H}oboken, {NJ}, 2001)}, volume 296 of {\em Contemp. Math.}, pages
  231--255. Amer. Math. Soc., Providence, RI, 2002.

\bibitem{[Wag99]}
J.~Wagner.
\newblock An algorithm for calculating the {N}ielsen number on surfaces with
  boundary.
\newblock {\em Trans. Amer. Math. Soc.}, 351(1):41--62, 1999.

\end{thebibliography}

\end{document}